\documentclass[20pt, twoside, a4paper]{article}

\let\Finishall\relax            
\ifx\TestIngCommand\undefined\relax                                     
\else\input ../../../proc/ppreamb-book.tex                               
\input init.tex \fi                                                     
\let\TestIngCommand\undefined                                             

\newtheorem{example}{Example}

\newtheorem{remark}{Remark}

\usepackage{amsmath,amscd,amssymb}
\usepackage{graphics}
\newtheorem{theo}{Theorem} 
\newtheorem{lem}{Lemma}

\newtheorem{defi}{Definition}
\newtheorem{proof}{Proof}

\newskip\ttglue\ttglue=.5em plus.25em minus.15em
\def\firstname#1{\def\FIRSTNAME{#1}\ignorespaces}                  
\def\lastname#1{\def\LASTNAME{#1}\ignorespaces}
\def\middleinitial#1{\def\MIDDLEINI{#1}\ignorespaces}
\def\department#1{\def\DEPARTMENT{#1}\ignorespaces}
\def\institute#1{\def\INSTITUTE{#1}\ignorespaces}
\def\address#1{\def\ADDRESS{#1}\ignorespaces}
\def\country#1{\def\COUNTRY{#1}\ignorespaces}
\def\otheraffiliation#1{\def\OTHERAFFILIATION{#1}\ignorespaces}
\def\email#1{\def\EMAIL{#1}\ignorespaces}
\newcount\autcount\autcount=0
\newcount\affcount\affcount=0
\newcount\numcount\newcount\nummcount
\newcount\nummmcount\newcount\nummmmcount
\def\writename#1#2{\ \kern-1ex\hbox{
  \csname AUthor\the#1\endcsname\
  \edef\TESTSTR{}\expandafter\ifx\csname auTHor\the#1\endcsname\TESTSTR
  \else\csname auTHor\the#1\endcsname.\ \fi 
  \csname authOR\the#1\endcsname$^{\csname AFF\the#1\endcsname}$
  \expandafter\ifx\csname corr\number#1\endcsname\relax
  \else\thanks{Corresponding author.}\ \fi
  }\ifnum#1<#2, \else\ \kern-1ex\fi}
\def\writeemail#1{
  \nummcount=0\relax\nummmcount=0\relax
  \loop\ifnum\nummcount<\autcount\advance\nummcount by1\relax
    {\expandafter\ifnum\csname AFF\the\nummcount\endcsname=#1\relax
    \global\advance\nummmcount by1\fi}\repeat
  \nummcount=0\relax\nummmmcount=0\relax
  \loop\ifnum\nummcount<\autcount\advance\nummcount by1\relax
    {\expandafter\ifnum\csname AFF\the\nummcount\endcsname=#1\relax
    \global\advance\nummmmcount by1\relax\def\blank{}\expandafter
    \ifx\csname EMAIL\the\nummcount\endcsname\blank(no e-mail)
    \else\csname EMAIL\the\nummcount\endcsname
    \fi 
    \ifnum\nummmmcount<\nummmcount; \fi\fi}\repeat}
\long\def\BeginAuthorList#1\EndAuthorList{#1\relax
  \author{\vbox{\hsize=390pt\noindent\numcount=0\relax
    \loop\ifnum\numcount<\autcount\advance\numcount by1\relax
      \writename{\numcount}{\autcount}
      \repeat}\\[2mm]
    \vbox{\small\numcount=0\relax
      \loop\ifnum\numcount<\affcount\advance\numcount by1\relax
        \vbox{{\count0=\numcount\relax
          \loop\expandafter\ifnum\csname AFF\the\count0\endcsname
            <\numcount\relax\advance\count0 by1\relax\repeat
          $^{\csname AFF\the\count0\endcsname}$}
        \def\BLANK{}\expandafter\ifx\csname DEPT\the\numcount\endcsname
          \BLANK
          \else\csname DEPT\the\numcount\endcsname, \fi
        \csname INST\the\numcount\endcsname,
        \csname ADDR\the\numcount\endcsname,
        \csname COUN\the\numcount\endcsname
        \edef\TEST{}\expandafter\ifx\csname OTHE\the\numcount\endcsname
          \TEST
          .\else;\break\csname OTHE\the\numcount\endcsname.\fi}
        \vbox{\writeemail{\numcount}}
        \repeat}\\}}
\expandafter\def\csname x1\endcsname{}
\expandafter\def\csname x2\endcsname{}
\expandafter\def\csname x3\endcsname{}
\expandafter\def\csname x4\endcsname{}
\expandafter\def\csname x5\endcsname{}
\expandafter\def\csname x6\endcsname{}
\expandafter\def\csname x7\endcsname{}
\expandafter\def\csname x8\endcsname{}
\expandafter\def\csname x9\endcsname{}
\def\Author#1#2{\global\advance\autcount by1\relax#2
  \expandafter\edef\csname AUthor\the\autcount\endcsname{\FIRSTNAME}
  \expandafter\edef\csname auTHor\the\autcount\endcsname{\MIDDLEINI}
  \expandafter\edef\csname authOR\the\autcount\endcsname{\LASTNAME}
  \expandafter\edef\csname EMAIL\the\autcount\endcsname{\EMAIL}
  \let\tempera\"\def\"{\string\"}\expandafter\ifx\csname x\DEPARTMENT
    \endcsname\relax
    \global\advance\affcount by1\relax\let\"\tempera
    \expandafter\edef\csname DEPT\the\affcount\endcsname{\DEPARTMENT}
    \expandafter\edef\csname INST\the\affcount\endcsname{\INSTITUTE}
    \expandafter\edef\csname ADDR\the\affcount\endcsname{\ADDRESS}
    \expandafter\edef\csname COUN\the\affcount\endcsname{\COUNTRY}
    \expandafter\edef\csname OTHE\the\affcount\endcsname{\OTHERAFFILIATION}
    \expandafter\edef\csname AFF\the\autcount\endcsname{\the\affcount}
  \else\expandafter\edef\csname AFF\the\autcount\endcsname{\DEPARTMENT}
  \fi\let\"\tempera\ignorespaces}
\def\CorrespondingAuthor#1#2{
  \expandafter\xdef\csname corr\number#1\endcsname{cor}
  \Author#1{#2}}
\def\PaperTitle#1{\title{\bf#1}}
\def\Category#1{\ignorespaces}
\def\keywords#1{{\noindent \emph{Keywords:}
  \def\BLANK{}\def\TEST{#1}\ifx\BLANK\TEST(n/a).\else#1\fi}}
\begin{document}

\large{                                                          
\PaperTitle{Bohr Almost Periodic Sets of Toral Type}

\Category{(Pure) Mathematics}

\date{}

\BeginAuthorList
  \Author1{
    \firstname{Wayne}
    \lastname{Lawton}
    \middleinitial{M}
    \department{Department of the Theory of Functions, Institute of Mathematics and Computer Science}
    \institute{Siberian Federal University}
    \otheraffiliation{}
   \address{Krasnoyarsk}
    \country{Russian Federation}
    \email{wlawton50@gmail.com}}
\EndAuthorList
\maketitle
\thispagestyle{empty}
\begin{abstract}
A locally finite multiset $(\Lambda,c),$ $\Lambda \subset \mathbb R^n, c : \Lambda \rightarrow \{1,...,b\}$ defines a Radon measure $\mu := \sum_{\lambda \in \Lambda} c(\lambda)\, \delta_\lambda$ that is Bohr almost periodic in the sense of Favorov if the convolution $\mu*f$ is Bohr almost periodic every $f \in C_c(\mathbb R^n).$ If it is of toral type: the Fourier transform $\mathfrak F \mu$ equals zero outside of a rank $m < \infty$ subgroup, then there exists a compactification $\psi : \mathbb R^n \rightarrow \mathbb T^m$ of $\mathbb R^n,$ a foliation of $\mathbb T^m,$ and a pair $(K,\kappa)$ where $K := \overline {\psi(\Lambda)}$ and $\kappa$ is a measure supported on $K$ such that $\mathfrak F \kappa = (\mathfrak F \mu) \circ \widehat \psi$ where $\widehat \psi : \widehat {\mathbb T^m} \rightarrow \widehat {\mathbb R^n}$ is the Pontryagin dual of $\psi.$  If $(\Lambda,c)$ is uniformly discrete Bohr almost periodic and $c = 1,$ we prove that every connected component of $K$ is homeomorphic to $\mathbb T^{m-n}$ embedded transverse to the foliation and the homotopy of its embedding is a rank $m-n$ subgroup $S$ of $\mathbb Z^m,$ and we compute the density of $\Lambda$ as a function of $\psi$ and the homotopy of comonents of $K.$ For $n = 1$ and $K$ a nonsingular real algebraic variety, this construction gives all Fourier quasicrystals (FQ) recently characterized by Olevskii and Ulanovskii and suggest how to characterize FQ for $n > 1.$ 
\end{abstract}
%
\footnote{\thanks{This work is supported by the Krasnoyarsk Mathematical 
Center and financed by the Ministry of Science and Higher Education 
of the Russian Federation in the framework of the establishment and 
development of Regional Centers for Mathematics Research and 
Education (Agreement No. 075-02-2020-1534/1).}}

\section{Preliminaries}\label{sec1}
This section defines notation and summarizes basic facts. Section \ref{sec2} 
introduces toral compactification concepts illustrated by examples based on the work of Favorov and Kolbasina \cite{favorov1,favorov2,favorov3,favorov4}, Kurasov and Sarnak \cite{kurasovsarnak}, Lagarias \cite{lagarias1,lagarias2}}, Lev, Olevskii and Ulanovskii \cite{levolevskii1,levolevskii2,levolevskii3,olevskiiulanovskii1,olevskiiulanovskii2}, and Meyer \cite{meyer1,meyer2,meyer3,meyer4,meyer5}. Section \ref{sec3} contains our main result: a detailed description of $m$--dimensional toral compactifications $K \subset \mathbb T^m$ of Bohr almost periodic Delone subsets $\Lambda \subset \mathbb R^n.$ We prove that the density of $\Lambda$ is determined by the homotopy classes of connected components of $K$ and the spectrum of $\Lambda$ is determined by a measure on $K.$ Section \ref{sec4} formulates research questions. 
\\ \\
$:=$ means `is defined to equal' and iff means `if and only if'.
$\mathbb N := \{1,2,3,...\},$ $\mathbb Z,$ $\mathbb Q,$ $\mathbb R,$ $\mathbb C$ are the natural, integer, rational, real, and complex numbers, $\mathbb T := 
\{z \in \mathbb C:|z| = 1\}.$ 
For $m, n \in \mathbb N$ and $A$ an integral domain, $A^{m \times n}$ is the set of $m$ by $n$ matrices over $A,$ and
$GL(m,A) \subset A^{m \times m}$ is its group of invertible matrices. $SO(n,\mathbb R)$ is the special orthogonal group. For a subgroup $S \subset \mathbb Z^m,$ $S_1 := \{k \in \mathbb Z^m: ak \in S \hbox{ for some } a \in \mathbb N\}$ is the smallest projective subgroup containing $S.$ There exists a unique subgroup $S_2$ with $S_1 \oplus S_2 = \mathbb Z^m.$
Smith's normal form theorem \cite{newman, smith} implies: (i) card $S_1/S < \infty,$ (ii) if $r := \hbox{rank } S = m,$ then $S_1 = \mathbb Z^m,$ (iii) if $r < m,$ there exists $K_1 \in \mathbb Z^{m \times r},$ $K_2 \in \mathbb Z^{m \times (m-r)}$ with $[K_1 \, K_2] \in GL(m,\mathbb Z),$ $S_1 = K_1\mathbb Z^{r},$
$S_2 := K_2\mathbb Z^{m-r}.$
\\
For $n \in \mathbb N,$ $\mathbb R^n$ is $n$--dimensional Euclidean space
consisting of column vectors, $B(0,R) \subset \mathbb R^n, R > 0$ is the open ball of radius $R$ with center $0$ and 
$V_n(R) := \frac{\pi^{\frac{n}{2}}}{\Gamma(\frac{n}{2}+1}\, R^n$ 
is its volume. We identify $\mathbb R^{m \times n}$ with the set of continuous homomorphisms from $\mathbb R^n$ to $\mathbb R^m,$ which equals the set of $\mathbb R$--linear maps from $\mathbb R^n$ to $\mathbb R^m.$  
$\rho_m : \mathbb R^m \rightarrow \mathbb T^m$ is the homomorphism 
$\rho_m(\,[x_1,...,x_m]^T\,) := [e^{2\pi i x_1},...,e^{2\pi i x_n}]^T.$
For $k = [k_1,...,k_m]^T \in \mathbb Z^m$
define $\zeta_k : {\mathbb T^m} \rightarrow \mathbb T$ by 
$\zeta_k(\, [z_1,...,z_m]^T \, ) := z_1^{k_1}\cdots z_m^{k_m}.$ 
Pontryagin's duality theorem (\cite{rudin}, Lemma 2.1.3) implies that if 
$G \subset \mathbb T^m$ is a closed subgroup, then its annihilator subgroup, defined by
$G^{\perp} := \{k \in \mathbb Z^m: \zeta_k(G) = \{1\} \},$ 
satisfies 
$G = \{g \in \mathbb T^m: \zeta_k(g) = 1 \hbox{ for every } k \in G^{\perp}\}.$ 
\\ \\
$C(\mathbb T^m)$ is the Banach algebra of continuous $f : \mathbb T^m \rightarrow \mathbb C$ with norm $||f|| := \max_{x \in \mathbb T^m} |f(x)|$ and $T(\mathbb T^m) := \hbox{span } \{\zeta_k:k \in \mathbb Z^m\}$ is its subalgebra of trigonometric polynomials.
The Stone--Weiestrass theorem (\cite{steinweiss}, Theorem 1.7) implies that $\overline {T(\mathbb T^m)} = C(\mathbb T^m).$
For $y \in \mathbb R^n$ define $\chi_y : {\mathbb R^n} \rightarrow \mathbb T$ by $\chi_y(x) := e^{2\pi i y^T\, x}.$ $C_b(\mathbb R^n)$ is the Banach algebra of bounded continuous $f : \mathbb R^n \rightarrow \mathbb C$ with norm $||f|| := \sup_{x \in \mathbb R^n} |f(x)|,$
$T(\mathbb R^n) := \hbox{span }\{\chi_y : y \in \mathbb R^n\}$ is its  subalgebra of trigonometric polynomials, and $U(\mathbb R^n) := \overline {T(\mathbb R^n)}$ is the Banach algebra of Bohr (or uniform) almost periodic functions \cite{bohr1,bohr2}. For $p \in [1,\infty),$ the completion of $T(\mathbb R^n)$ with respect to the norm
\begin{equation}\label{Bes}
	 ||f||_{B,p} := \left( \limsup_{R \rightarrow \infty} \frac{\int_{B(0,R)} |f|^p}{V_n(R)}\right)^\frac{1}{p}
\end{equation}
is the Banach space $B^p(\mathbb R^n)$ of Besicovitch almost periodic functions \cite{besicovitch1,besicovitch2}. $U(\mathbb R^n) \subset B^p(\mathbb R^n),$ and $p > q$ implies $B^p(\mathbb R^n) \subset B^q(\mathbb R^n).$ For $f \in B^1(\mathbb R^n)$ the following limit exists
\begin{equation}\label{hat}
		 \widehat f\, (y) := \lim_{R \rightarrow \infty} \frac{\int_{B(0,R)} f \, \chi_{-y}}{V_n(R)}, \ \ y \in \mathbb R^n
\end{equation}
and  defines a function $\widehat f : \mathbb R^n \rightarrow \mathbb C$ that is nonzero only on  a countable set $\Omega(f)$ called the spectrum of $f.$ This gives the formal expansion
\begin{equation}\label{formexpf}
	f \sim \sum_{y \in \Omega(f)} \widehat f\, (y)\, \chi_y.
\end{equation}
Every $f \in B^2(\mathbb R^n)$ satisfies Parseval's equation
$
	||f||_{B,2}^2 = \sum_{y \in \Omega(f)} |\widehat f\, (y)|^2,
$
and every square summable $g : \mathbb R^n \rightarrow \mathbb C$ equals $\widehat f$ for some $f \in B^2(\mathbb R^n).$
$\Omega(f)$ is the smallest set such that
$f$ can be approximated by elements in span $\{\, \chi_y \, : \, y \in \Omega(f)\, \}.$ 
If the space $C_c(\mathbb R^n)$ is equipped with the inductive limit topology (\cite{treves}, p. 515), then its dual is the space of Radon measures (\cite{treves}, p. 217--222). The Fourier transform $\mathfrak F : L^1(\mathbb R^n) \rightarrow C_0(\mathbb R^n)$ is defined by $(\mathfrak F \, f)(y) := \int_{\mathbb R^n} f\, \chi_{-y}.$
$\mathcal S(\mathbb R^n)$ is the Schwartz space \cite{schwartz}, consisting of smooth functions $f : \mathbb R^n \rightarrow \mathbb C$ all of whose derivatives vanish faster than $(1+||x||)^{-N}$ for every $N \in \mathbb N$ and equipped with the appropriate Fr\'{e}chet space topology. The Fourier transform $\mathfrak{F} : S(\mathbb R^n) \rightarrow S(\mathbb R^n)$ is a bijection. The dual space $\mathcal S^{\, \prime}(\mathbb R^n)$ is the Fr\'{e}chet space of tempered distributions 
(\cite{treves}. p. 272, Definition 25.2). The Fourier transform extends to give a bijection
$\mathfrak F : \mathcal S^{\, \prime}(\mathbb R^n) \rightarrow \mathcal S^{\, \prime}(\mathbb R^n)$
\begin{equation}\label{FTdist}
	(\mathfrak F \, \eta)(f) := \eta(\mathfrak F \, f), \ \ \eta \in \mathcal S^{\, \prime}(\mathbb R^n), \ f \in S(\mathbb R^n).
 \end{equation}
\begin{defi}\label{def1}
A Radon measure $\mu$ on $\mathbb R^n$ is Bohr; Besicovitch almost periodic if convolution with it satisfies 
$\mu*\, C_c(\mathbb R^n) \subset U(\mathbb R^n); B^p(\mathbb R^n).$ A tempered distribution $\mu$ on 
$\mathbb R^n$ is Bohr;Besicovitch almost periodic if 
$\mu* S(\mathbb R^n) \subset U(\mathbb R^n); B^p(\mathbb R^n).$ 
An atomic measure is a Radon measure $\mu$ such that there exists $\Lambda \subset \mathbb R^n$ and $c : \Lambda \rightarrow \mathbb C$ such that $\mu = \sum_{\lambda \in \Lambda} c(\lambda)\, \delta_\lambda,$ where $\delta_\lambda$ is the point measure at $\lambda.$ 
A crystalline measure is an atomic measure $\mu$ that is also a tempered distribution whose Fourier transform is also an atomic measure, so $\mathfrak{F}(\mu) = \sum_{\omega \in S} a(\omega) \, \delta_{\omega}$ is an atomic measure, and both $\Lambda$ and $\Omega(\mu)$ have no limit points (locally finite). $\mu$ is then a Fourier quasicrystal if both 
$|\mu| := \sum_{\lambda \in \Lambda} |c(\lambda| \delta_\lambda,$ and
$|\mathfrak{F}(\mu)| := \sum_{\omega \in S} |a(\omega)| \, \delta_{\omega}$ are tempered distributions.
\end{defi}
\begin{remark}\label{rem1}
Meyer proved that crystalline measures are Bohr almost periodic tempered distributions (\cite{meyer4}, Lemma 2), but not necessarily Bohr almost periodic measures (\cite{meyer4}, Theorems 4 and 5).
\end{remark}
If $\mu$ is an atomic measure and $\mu*C_c(\mathbb R^n) \subset B^1(\mathbb R^n)$ then the limit
\begin{equation}\label{hatmu}
	\widehat \mu\, (y) := \lim_{R \rightarrow \infty} \frac{\sum_{\lambda \in \Lambda \cap B(0,R)} c(\lambda) \, \chi_{-y}(\lambda)}{V_n(R)}, \ \ y \in \mathbb R^n
\end{equation}
exists and equals $0$ except at a countable set called the spectrum $\Omega(\mu)$ of $\mu.$ This gives the following formal expansions:
\begin{equation}\label{formexpmu}
	\mu \sim \sum_{y \in \Omega(\mu)} \widehat \mu\,(y) \, \chi_{y}, \ \ \ \
	\mathfrak F \mu \sim \sum_{y \in \Omega(\mu)} \widehat \mu \, (y)\, \delta_{y}.
\end{equation}
A multiset is a pair $(\Lambda,c),$ where $\Lambda \subset \mathbb R^n$ is nonempty and $c : \mathbb R^n \rightarrow \mathbb N$ satisfies $c(\lambda) > 0$ iff $\lambda \in \Lambda.$ They form a commutative semigroup under the binary operation $(\Lambda_1,c_1)+(\Lambda_2,c_2) := (\Lambda,c)$ where $\Lambda := \Lambda_1 \cup \Lambda_2$ and $c :=c_1 + c_2.$ A $\Lambda$ is a multiset $(\Lambda,1$ where $\Lambda$ is a tranlate of a discrete rank $n$ subgroup of $\mathbb R^n.$ A multiset is trivial if it is the sum of lattices. 
\begin{equation}\label{R1}
	R_1(\Lambda) := \inf \left\{||\lambda_2-\lambda_1||:\lambda_1,\lambda_2 \in \Lambda,\, \lambda_1 \neq \lambda_2\right\} \in [0,\infty),
\end{equation}
\begin{equation}\label{R2}
	R_2(\Lambda) := \inf \{R \in [0,\infty] : \bigcup_{\lambda \in
\Lambda} B(\lambda,R) = \mathbb R^n\}.
\end{equation}
\begin{equation}\label{eqb}
	b(\Lambda,c) := \lim_{R \rightarrow 0} \sup_{x \in \mathbb R^n} \sum_{\lambda \in \Lambda \cap B(x,R)} c(\lambda) \in \mathbb N \cup \{\infty\},
\end{equation}
\begin{defi}\label{def2} $\Lambda \subset \mathbb R^n$ is uniformly discrete; syndetic if $R_1(\Lambda) > 0; R_2(\Lambda) < \infty,$ Delone (aka Delaunay) \cite{delone} if it is uniformly discrete and syndetic, finite type if there exists a finite $F \subset \mathbb R^n$ with $\Lambda - \Lambda \subset \Lambda + F,$ quasiregular \cite{galiulin} if both $\Lambda$ and $\Lambda - \Lambda$ are Delone, and a Meyer set \cite{moody} if it is Delone and finite type. A multiset $(\Lambda,c)$ is uniformly finite if $b(\Lambda,c) < \infty.$
\end{defi}
\begin{remark}\label{rem2}
Lagarius proved (\cite{lagarias1}, Theorem 1.1) that $\Lambda$ is quasiregular iff it is a Meyer set.
\end{remark}
A uniformly finite multiset $(\Lambda,c)$ defines a Radon measure 
\begin{equation}\label{mu}
	\mu(\Lambda,c) := \sum_{\lambda \in \Lambda} c(\lambda)\, \delta_\lambda.
\end{equation}
Then $\mu := \mu(\Lambda,c)$ defines, by convolution, the map
$\mu* : C_c(\mathbb R^n) \rightarrow C_b(\mathbb R^n)$
\begin{equation}\label{TLambda}
	(\mu*f)(x) = \sum_{\lambda \in \Lambda} c(\lambda)\, f(x-\lambda), \ \ f \in C_c(\mathbb R^n),\, x \in \mathbb R^n
\end{equation}
$\mu$ is also a tempered distribution and
$
	\mathfrak F (\mu*f) = (\mathfrak F \mu)\,(\mathfrak F \,f), \ \ f \in C_c(\mathbb R^n).
$
\begin{defi}\label{def3}
A uniformly finite multiset $(\Lambda,c)$ is Bohr; Besicovitch almost periodic if $\mu*\, C_c(\mathbb R^n) \subset U(\mathbb R^n); B^p(\mathbb R^n).$ We define it to be a Fourier quasicrystal if $\mu := \mu(\Lambda,c)$ is a quasicrystal. 
\end{defi}
If $(\Lambda,c)$ is trivial, then for every $f \in C_c(\mathbb R^n,$ $\mu*f$ is a sum or periodic functions hence $(\Lambda,c)$ is Bohr almost periodic. 
\begin{remark}\label{rem3}
Lagarias (\cite{lagarias2}, Problem 4.4) asked if every Bohr almost periodic multiset is trivial. Favorov answered no by proving that every Bohr almost periodic perturbation of a lattice is Bohr almost periodic \cite{favorov1,favorov2}. He also proved that every Bohr almost periodic Meyer set is trivial \ref{ex1}, (\cite{favorov3}, Theorem 2). 
\end{remark}
The Poisson summation formula (\cite{steinweiss}, Theorem 2.4) implies that trivial multisets are Fourier quasicrystals. Olevskii and Ulanovskii extended work with Lev \cite{levolevskii1,levolevskii2,levolevskii3} and work of Meyer \cite{meyer3,meyer4} and Kurasov and Sarnak \cite{kurasovsarnak} by proving the following deep result (\cite{olevskiiulanovskii2}, Theorem 8):
\begin{theo}\label{thm1}
For $n = 1,$ $(\Lambda,c)$ is a Fourier quasicrystal iff there exists a trigonometric polynomial whose extension to an entire function only has real roots $\Lambda$ with multiplicities $c.$ 
\end{theo}
For a trigonometric polynomial $p$ let $Z_r(p)$ be its set of (real) roots, 
$Z_c(p)$ the set of complex roots of its entire extension $P,$ and 
$\rho_r(p);\rho_c(p)$ their densities (counted with multiplicities). 
The following result characterizes trigonometric polynomials 
whose entire extensions have only real roots.
\begin{theo}\label{thm2}
Let 
$p(x) := \sum_{j=1}^d c(d)\chi_{y_j},$ $d \geq 2,$ $y_1 < \cdots < y_d,$ $c(j) \in \mathbb C \backslash \{0\},$ be a trigonometric polynomial.
Then $Z_r(p) = Z_c(p)$ iff $\rho_r(p) = \rho_c(p).$  Furthermore, $\rho_c(p) = y_d-y_1.$
\end{theo}
\begin{proof}
Clearly $Z_r(p) \subset Z_c(p)$ so $\rho_r(p) \leq \rho_c(p).$ If $P$ had a nonreal root $z$ then there would exists a circle $C$ centered at $z$ such that $C \cap \mathbb R = \phi,$ $P$ has only the root $z$ inside $C,$ $|P|$ has a minimum value $\epsilon > 0$ on $C,$ and the winding number of $P$ on $C$ equals the multiplicity of the root $z.$ Choose $B > 0$ such that $|y| < B$ for all $x+iy \in C.$ Since $P(x+iy)$ is almost periodic in $x$ uniformly for $|y| \leq B,$ there exists a syndetic set $S \subset \mathbb R$ such that $|P(x+iy) - P(x+s+iy)| < \epsilon$ for all $x+iy \in C$ and $s \in S.$ By Rouche's theorem $P$ has a root inside the circle $C+s$ for every $s \in S.$ Since $S$ is syndetic, the density of non real zeros is $> 0$ hence $Z_c(p) > \rho(Z(p).$ This implies the first assertion. The second assertion is a classic result by Titschmarsh (\cite{titchmarsh}, Theorem 1). 
\end{proof}
\section{Toral Compactifications of $\mathbb R^n$ and $(\Lambda,c)$}\label{sec2}
Let $\mathbb R^n_d$ be the group $\mathbb R^n$ equipped with the discrete topology. Its Pontryagin dual is  the Bohr compactification $B(\mathbb R^n)$ 
(\cite{rudin}, 1.8). Bohr showed that $U(\mathbb R^n)$ corresponds to $C(B(\mathbb R^n))$ \cite{bohr1,bohr2}, and Besicovitch showed that $B^p(B(\mathbb R^n))$ corresponds to $L^p(B(\mathbb R^n))$ \cite{besicovitch1,besicovitch2}. The Bohr compactification is very large, in particular, it is not first countable. Many almost periodic functions and related concepts can be studied using smaller compactifications. These include trigonometric polynomials and the model sets pioneered by Meyer \cite{meyer1, meyer2} in the context of harmonic analysis and number theory and later discovered empirically as a class of quasicrystals by Shechtman and colleagues \cite{schectman}.
\begin{defi}\label{toralcomp}
A toral compactification of dimension $m$ of $\mathbb R^n$ is a pair $(\mathbb T^m,\psi)$ where $\psi : \mathbb R^n \rightarrow \mathbb T^m$ is a continuous homomorphism with dense image.
\end{defi}
A toral compactification defines an ergodic action of $\mathbb R^n$ on $\mathbb T^n$
$
	x(z) := \psi(x)z, \ \ x \in \mathbb R^n, \, z \in \mathbb T^m
$
and an associated foliation whose leaves are the orbits of this action.
%
%
\begin{lem}\label{lem1}
The toral compactifications of dimension $m$ of $\mathbb R^n$ coincide with the set of pairs $(\mathbb T^m,\psi)$ where $\psi = \rho_m \circ M$ $($composition$)$ and $M \in \mathbb \mathbb R^{m \times n}$ has rows that are linearly independent over $\mathbb Q.$
\end{lem}
\begin{proof}
Since $\rho_m : \mathbb R^m \rightarrow \mathbb T^m$ is a covering space, 
$\mathbb R^n$ is path, locally path and simply connected, 
the lifting property \hbox{(\cite{hatcher}, Theorem 1.55)} implies that there exists 
a unique continuous map $h : \mathbb R^n \rightarrow \mathbb R^m$ such that 
$\psi = \rho_m \circ h$ and $h(0) = 0.$ Define 
$\eta : \mathbb R^n \times \mathbb R^n \rightarrow \mathbb R^m$ by $\eta(x,y) = h(x+y)-h(x)-h(y).$ Since $\psi$ and $\rho_m$ are homomorphisms, 
$\rho_m \circ \eta = 1$ hence the image of $\eta$ is a subset of $\mathbb Z^m = \hbox{kernel }\rho_m.$ Since $\eta$ is continuous and its domain is connected, its image is connected. Since $\eta(0,0) = 0,$ $\eta = 0,$ $h$ is a homomorphism so there exists $M \in \mathbb R^{m \times n}$ such that $h = M,$ hence $\psi = \rho_m \circ M.$ Then $\overline {\psi(\mathbb R^n)} \neq \mathbb T^m$ iff there exists $k \in \mathbb Z^m \backslash \{0\}$ with $1 = \zeta_k\circ \psi = \chi_{M^Tk} \Leftrightarrow M^Tk = 0$ iff the rows of $M$ are not linearly independnet over $\mathbb Q.$
\end{proof}
%
%
Define $\Psi : C(\mathbb T^m) \rightarrow C_b(\mathbb R^n)$ by
$\Psi g := g \circ \psi,$ and 
$\Omega_\psi := M^T\, \mathbb Z^m.$
\begin{lem}\label{lem2}
$\Psi T(\mathbb T^m) = \hbox{span }\{\, \chi_x\, :\, x\in \Omega_\psi\, \}$ and
$\Psi$ is an isometry.
$\Psi\, C(\mathbb T^m) = 
\{\, f \in U(\mathbb R^n)\, :\, \Omega(f) \subset \Omega_\psi\, \}.$ For $p \in [1,\infty),$
$\Psi$ extends to a linear isometry $\Psi : L^p(\mathbb T^m) \rightarrow B^p(\mathbb R^n)$ and
$\Psi\, L^p(\mathbb T^m) =
\{\, f \in B^p(\mathbb R^n)\, : \, \Omega(f) \subset \Omega_\psi\, \}.$
\end{lem}
\begin{proof}
The first assertion is obvious, the second follows since $\psi$ has a dense image.
If $g \in C(\mathbb T^m),$ the Stone--Weierstrass theorem implies there exists a sequence 
$P_n \in T(\mathbb T^m)$ with 
$||P_n -g|| \rightarrow 0.$ 
Define 
$Q_n = \Psi P_n$ and 
$f := \Psi g.$ Then $Q_n \in \hbox{span }\{\, \chi_x\, :\, x\in G_\psi\, \}$ and since $\Psi$ is an isometry,
$||Q_n - f|| \rightarrow 0$ so $f \in U(\mathbb R^n)$ and $\Omega(f) \subset G_\psi.$ Conversely, if $f \in U(\mathbb R^n)$ and $\Omega(f) \subset \Omega_\psi$ then there exists $Q_n \in \hbox{span }\{\, \chi_x\, :\, x\in \Omega_\psi\, \}$ with
$||Q_n - f|| \rightarrow 0.$ Then there exists a sequence $P_n \in T(\mathbb T^m)$ with $\Psi P_n = Q_n.$ Since $Q_n$ is a Cauchy sequence and $\Psi$ is a isometry, $P_n$ is a Cauchy sequence in the complete metric space $C(\mathbb T^m)$ so there exists $g \in C(\mathbb T^m)$ such that $||P_n-g|| \rightarrow 0.$ Therefore $f = \Psi g.$
The second equality is proved in a similar way using the $||\, ||_{B,p}$ norm. 
\end{proof}
%
%
\begin{defi}\label{APS2}
A multiset $(\Lambda,c)$ is of toral type if it is Besicovitch almost periodic and $\Omega(\mu)$ generates a rank $m < \infty$ subgroup of $\mathbb R^n.$  
\end{defi}
Lemma \ref{lem2} implies that $(\Lambda,c)$ is of toral type iff there exists a compactification $(\mathbb T^m,\psi)$ with $\mu*C_c(\mathbb R^n) \subset \Psi \, L^p(\mathbb T^m).$ Then define $\Theta : C_c(\mathbb R^n) \rightarrow L^p(\mathbb T^m)$ so
\begin{equation}\label{Theta}
	\Psi \circ \Theta = \mu* : C_c(\mathbb R^n) \rightarrow B^p(\mathbb R^n),
\end{equation}
and observe that $(\Lambda,c)$ is Bohr almost periodic iff $\Theta C_c(\mathbb R^n) \subset C(\mathbb T^m).$
%
%
\begin{lem}\label{lem3} 
If a multiset $(\Lambda,c)$ is Bohr almost periodic of toral type and $\Lambda$ is Delone then for every $\lambda_1 \in \Lambda$ and $\gamma \in$ kernel $\psi,$ there exists
$\lambda_2 \in \Lambda$ with $c(\lambda_2) = c(\lambda_1)$ and 
$\lambda_1 + \gamma = \lambda_2.$
Therefore $r := $rank kernel $\psi \leq n$ and rank $M = n.$ If $r = n$ then $\Lambda$ is trivial.
\end{lem}
\begin{proof}
Let $f \in C_c(\mathbb R^n)$ be a nonzero functions with support $f \subset B(0,R_1(\Lambda)/2).$
Then since $\mu*f = (\Theta f)\circ \psi,$ the function $\mu*f$ is invariant under translation by elements in kernel $\psi.$ Therefore for every $\gamma \in \hbox{kernel }\psi,$ 
$$\sum_{\lambda \in \Lambda} c(\lambda)\, f(x+\lambda + \gamma) = \sum_{\lambda \in \Lambda} c(\lambda) \, f(x+\lambda), \ \ \ x \in \mathbb R^n.$$
Since the supports of the summands on both sides of this equation are disjoint, for every $\lambda_1 \in \Lambda$ there exists $\lambda_2 \in \Lambda$ such that
$$
c(\lambda_1)\, f(x+\lambda_1 + \gamma) = c(\lambda_2)\, f(x + \lambda_2), \ \  \ x \in \mathbb R^n.
$$
This proves the first assertion. Since $\Lambda$ is discrete, kernel $\psi$ is discrete hence $r \leq n$ and rank $M = n.$  If $r = n$ then $\Lambda$ consists of orbits in $\Lambda$ of kernel $\psi.$ Since $\Lambda$ is discrete there are a finite number of orbits so $\Lambda$ is trivial.
\end{proof}
%
%
\begin{defi}\label{comlambda}
Let $(\Lambda,c)$ be a Besicovitch almost periodic multiset of toral type with associated compactification $(\mathbb T^m,\psi)$ of $\mathbb R^n.$ 
The compactification of $(\Lambda,c)$ is the pair $(K,\kappa)$ where
\begin{equation}\label{K}
	K := \overline {\psi(\Lambda)}
\end{equation}
is the closure of $\psi(\Lambda),$ and $\kappa : C(\mathbb T^m) \rightarrow \mathbb C$ is the measure defined by 
\begin{equation}\label{kappa}
	\kappa \,  g := \lim_{R \rightarrow \infty} \frac{\sum_{\lambda \in \Lambda \cap B(0,R)} c(\lambda) g(\psi(\lambda))}{V_n(R)}, \ \ \ g \in C(\mathbb T^m).
\end{equation}
\end{defi}
The existence of the limit in Equation \ref{kappa} follows from Equation \ref{hat} with $y = 0$ as follows. Let $h_n \in C_c(\mathbb R^n)$ be an approximate identity sequence
and $\mu = \sum_{\lambda \in \Lambda} c(\lambda)\delta_\lambda$ be the Radon measure associated with the Besicovitch almost periodic multiset $(\Lambda,c).$
Then $\mu*h_n \in B^1(\mathbb R^n)$ and $g \circ \psi \in U(\mathbb R^n)$ so
the product $g_n := (\mu*h_n)(g \circ \psi) \in B^1(\mathbb R^n)$ hence
$\widehat g_n(0)$ exists. Since $g \circ \psi$ is uniformly continuous 
$\lim_{n \rightarrow \infty} \widehat g_n(0) = \kappa g.$ Furthermore 
since $||\mu*h_n||_{B,1}$ are uniformly bounded by some constant $C > 0,$ it follows that $|\widehat g_n(0)| \leq C||g \circ \psi|| = C\max_{z \in \mathbb T^m} |g(z)|$ hence $|\kappa g| \leq C \max_{z \in \mathbb T^m} |g(z)|.$ Therefore
the Riesz-Markov-Kakutani theorem implies that  $\kappa$ is a Borel measure on $\mathbb T^m.$ Clearly $\kappa$ is positive, supported on $K,$ and satisfies
\begin{equation}\label{FTkappa}
\kappa \, \zeta_{-k} = \widehat \mu(M^Tk), \ \ \ k \in \mathbb Z^m.
\end{equation}
Let $\sigma_1$ be Lebesgue measure on $\mathbb R^n$ and $\sigma_2$ be Haar measure on $\mathbb T^m.$ Then $f \in C_c(\mathbb R^n) \implies f\sigma_1$ is a measure on $\mathbb \mathbb R^n$ hence $\psi$ induces the measure $\psi^*(f\sigma_1)$ on $\mathbb T^m$ and
\begin{equation}\label{conv}
	\psi^*(f\sigma_1)*\kappa = (\Theta f)\sigma_2.
\end{equation}
Lemma\ref{lem3} implies that $M^TM \in \mathbb R^{n \times n}$ is symmetrix and positive definite. Therefore there exists $O \in SO(n,\mathbb R)$ with $(MO)^TMO = D^2$ where $D$ is a diagonal matrix with positive diagonal entries, so  $MOD^{-1}$ has orthonormal columns. Henceforth we assume that a change of variables has been made that transforms a multiset $(\Lambda, c)$ into $(DO^T\Lambda,c \circ OD^{-1})$ having the same properties described in Definitions \ref{def2} and \ref{def3}, and transforms $M$ to satisfy $M^TM = I_n$ (the $n$ by $n$ identity matrix). 
Construct $N \in \mathbb R^{m \times (m-n)}$ such that
$[M \, N] \in SO(m,\mathbb R),$ the $(m-n)$--form
\begin{equation}\label{eq:xiN}
\xi_N := (2\pi i)^{-(m-n)} \sum_{1 \leq i_1,...,i_{m-n} \leq m} 
N_{i_{1,1}}\cdots N_{i_{{m-n},m-n}}\, 
\frac{dz_{i_1}}{z_{i_1}} 
\wedge \cdots \wedge 
\frac{dz_{i_{m-n}}}{z_{i_{m-n}}},
\end{equation}
and the $n$--form $\xi_M$ from $M.$ $\xi_M; \xi_N$ is determined by $MM^T;NN^T.$ Furthermore 
$
\xi_M \wedge \xi_N = \hbox{volume form on } \mathbb T^m.
$
%
%
\begin{lem}\label{lem4}
If $K$ is an $(m-n)$-dimensional submanifold of $\mathbb R^m,$ then
$\kappa = |\kappa_c|,$ (the total variation measure), where $\kappa_c$ is the complex Borel measure
\begin{equation}\label{kappaext}
	\kappa_c \, g := \int_{K} g \, \xi_N, \ \ g \in C(\mathbb T^m),
\end{equation}
where the connected components of $K$ are given arbitrary orientations ($\kappa_c$ depends on the orientations but $\kappa$ does not).
\end{lem}
\begin{proof}
Follows from Definition \ref{comlambda} and the ergodic action of $\mathbb R^n$ on $\mathbb T^m.$
\end{proof}
%
%
For Laurent polynomials $P_j : (\mathbb C \backslash \{0\})^m \rightarrow \mathbb C, \, j = 1,...,q$ define
$$Z_c(P_1,...,P_q) := \left\{ \, z \in (\mathbb C \backslash \{0\})^m : P_j(z) = 0, \ , j = 1,...,q\, \right\},$$
and $Z_r(P_1,...,P_q) = Z_c(P_1,...,P_q) \cap \mathbb T^m.$
\begin{remark}\label{rem4}
If $K = Z_r(P_1,...,P_q)$ then $\widehat \kappa : \mathbb Z^m \rightarrow \mathbb C$ satisfies a system of $q$--linear difference equations having constant coefficients. 
\end{remark}
Let $D := \{z \in \mathbb C:|z| < 1\}.$ A Laurent polynomial is self--dual if
$P(z^{-1})/P(z)$ is a monomial, and stable if $D^m \cap Z(P) = \phi.$ Lee-Yang polynomials are self-dual and stable \cite{leeyang}. We describe three toral compactifications where $n = 1, m = 2,$ $\theta \in [-\pi,\pi],$ $M = [\cos \theta, \sin \theta]^T,$ and $\psi := \rho_2 \circ M.$
\begin{example}\label{ex1}
This example belongs to a class of Fourier quasicrystals consructed by Meyer (\cite{meyer5}, Definition 6) using Ahern measures \cite{ahern}.
Let $P$ be the Lee-Yang polynomial $2z_1z_2 + z_1 + z_2 + 2$ and
$\Lambda = \psi^{-1}(Z_r(P)).$ 
Then
$$K := \overline \Lambda = Z_r(P) = \left\{ \, [z_1, z_2]^T \in \mathbb T^2: z_1 \in \mathbb T, \, z_2 = -\frac{2+z_1}{1+2z_1}\, \right\}$$
is homeomorphic to $\mathbb T,$ and $\Lambda$ is the set of roots of the trigonometric polynomial $T = 2 \chi_{\cos \theta + \sin \theta} + \chi_{\cos \theta} + \chi_{\sin \theta} + 2,$ which is the set of real roots of its entire extension.
The density of complex zeros of its entire extension equals $\max \{\cos \theta + \sin \theta, \cos \theta, \sin \theta, 0\} - \min \{\cos \theta + \sin \theta, \cos \theta, \sin \theta, 0\}.$ This equals $|\cos \theta + \sin \theta|$ if $\cos \theta$ and $\sin \theta$ have the same sign, and equals $\max \{|\cos \theta|, |\sin \theta|\}$ otherwise. The density of $\Lambda$ equals $|\cos \theta + \sin \theta|,$ therefore Theorems \ref{thm1} and \ref{thm2} imply that $\Lambda$ is a quasicrystal iff $\tan \theta \in [0,\infty].$  
We now derive this conclusion by computing the Fourier transform of $\kappa.$ 
Let $z_j = e^{i\theta_j}, j = 1,2$ and orient $Z_r(P)$ by $\theta_1$ increasing from $0$ to $2\pi.$ Then
$\frac{d\theta_2}{d\theta_1} = \frac{-3}{5+4\cos \theta_1}$ so for $\tan \theta \notin [-3,-1/3],$ $Z_r(P)$ is transverse to the foliation, $\xi_N = (-\sin \theta)d\theta_1 + (\cos \theta) d\theta_2,$ so Equation \ref{FTkappa} implies
$$\widehat \mu([k_1, k_2]\, M) = \kappa \, \zeta_{-[k_1, k_2]^T} = -\int_{K} z_1^{-k_1}z_2^{-k_2} \xi_N = 0$$
whenever $k_1k_2 < 0.$ This implies that $\Omega(\Lambda) = $ projection of the support of $\widehat \kappa$ onto the line $M\, \mathbb R$ is discrete whenever $\tan \theta \in (0,\infty) \backslash \mathbb Q,$ and then 
$\Lambda$ is a FQ. Note that $\Lambda$ is trivial iff $\tan \theta \in \mathbb Q \cup \{\infty\}.$ 
$\widehat \kappa$ satisfies the difference equation
\begin{equation}\label{ex1diff}
2\widehat \kappa ([k_1, k_2]^T) + \widehat \kappa ([k_1-1, k_2]^T)  +\widehat \kappa ([k_1, k_2-1]^T) + 2\widehat \kappa ([k_1-1, k_2-1]^T)  = 0.
\end{equation}
\end{example}
%
%
\begin{example}\label{ex2}
This example includes as a special case one constructed by Olevskii and Ulananovskii (\cite{olevskiiulanovskii1}, Example 1) 
Let $P(z_1,z_2) := z_1 - z_1^{-1} - \delta(z_2 - z_2^{-1})$ Note that $Z_r(P)$ has two connected components that approach the circles $z_1 = 1$ and $z_1 = -1$ as $\delta \rightarrow 0.$ Let $\Lambda := \psi^{-1}(Z_r(P)).$
Then $K := \overline \Lambda = Z_r(P),$ and $\Lambda$ 
is the set of roots of the trigonometric polynomial 
$T := \chi_{\cos \theta} - \chi_{-\cos \theta} - \delta(\chi_{sin \theta} - \chi_{-\sin \theta}).$ 
For $|\tan \theta| < 1/\delta,$ the folliation is transverse to $K$ and the density of $\Lambda$ equals $2 |\cos \theta|.$ A more detailed geometric analysis shows that if $|\tan \theta| \geq 1/\delta$ then the density is less than $2|\sin \theta.$ The density of complex zeros of the entire extension $T$ equals 
$2 \, \max \{|\cos \theta|, \, |\sin \theta|.$ Therefore Theorems \ref{thm1} and \ref{thm2} imply that 
$\Lambda$ is a FQ iff $\tan \theta \in (0,1) \backslash \mathbb Q.$ Equation \ref{FTkappa} implies
$$\widehat \mu([k_1, k_2]\, M) = \kappa \, \zeta_{-[k_1, k_2]^T} = -\int_{K} z_1^{-k_1}z_2^{-k_2} \xi_N = 0$$
whenever $k_2 > -k_1 \geq 1$ or $k_2 < -k_1 \leq 0.$ This implies that $\Omega(\Lambda) = $ projection of the support of $\widehat \kappa$ onto the line $M\, \mathbb R$ is discrete whenever $\tan \theta \in (0,1) \backslash \mathbb Q,$ and then 
$\Lambda$ is a FQ.
$\widehat \kappa$ satisfies the difference equation
\begin{equation}\label{ex2diff}
\widehat \kappa ([k_1+, k_2]^T) - \widehat \kappa ([k_1-1, k_2]^T)  
-\delta \widehat \kappa ([k_1, k_2+1]^T) +\delta \widehat \kappa ([k_1, k_2-1]^T)  = 0.
\end{equation}
\end{example}
%
%
\begin{example}\label{ex3}
Let $\tan \theta \notin \mathbb Q \cup \{\infty\}$ and $\ell > 0.$ Then
$$\Lambda := \{m\cos \theta + n\sin \theta: m, n \in \mathbb Z, \,
|m\sin \theta - n\cos \theta| < \ell/2\}$$
is a cut and project set. It is a Meyer subset of $\mathbb R$ and nontrivial so Remark \ref{rem3} implies that it is not Bohr almost periodic.
$K = \rho_2(\{t[-\sin \theta, \cos \theta]^T:t \in (-\ell/2,\ell/2)\})$ and
$$\kappa \, g = \int_{-\ell/2}^{\ell/2} g\circ \rho_2(t[-\sin \theta,\cos \theta]^T)\, dt, \ \ g \in C(\mathbb T^2),$$ 
so Equation \ref{conv} implies that for every $f \in C_c(\mathbb R),$
$\Theta\, f \in L^p(\mathbb T^2)$ so $\mu*f \in B^p(\mathbb R)$ and $\Lambda$ is Besicovitch almost periodic. But $\Theta f$ has discontinuities on leaves of the foliation that intersect the boundary of $K$ so $\mu*f \notin U(\mathbb R).$
$\kappa\, \zeta_k = \hbox{sinc }(\pi \ell [-\sin \theta \cos \theta]k), \, k \in \mathbb Z^2$ hence $\Omega(\mu) = \cos \theta \, \mathbb Z + \sin \theta \, \mathbb Z$ is a dense subset of $\mathbb R$ so $\Lambda$ is not a Fourier quasicrystal. We observe that if $\alpha := \frac{1+\sqrt 5}{2}, \tan \theta = \alpha,$ and $A := [0, 1;1, 1] \in GL(2,\mathbb Z),$ then
\begin{equation}\label{dilate}
\psi(\alpha x) = A\, \psi(x), \ \ x \in \mathbb R,
\end{equation}
therefore $\alpha \, \Lambda \subset \Lambda,$ so $\Lambda$ is quasicrystal which admits a dilation.
\end{example}
%
%
\section{Homotopy Class of $K$ and Density of $\Lambda.$}\label{sec3}
Assume that $\Lambda \subset \mathbb R^n$ is a Delone Bohr almost periodic set of toral type with associated compactification $(\mathbb T^m,\psi)$ of $\mathbb R^n$ and compactification $(K,\kappa)$ of $\Lambda$ ($c = 1$). 
We assume that a change of variables has been made so that $M^TM = I_n$ where $\psi = \rho_m \circ M.$ If $m = n$ then $\Lambda$ is trivial, $K$ is finite, and $\kappa$ is counting measure on $K.$ We henceforth assume that $m > n.$ We will derive a precise description of $K.$
\\ \\
Define $H := \rho_m^{-1}(K),$  and $V := M\mathbb R^n,$ so
$\overline {M\Lambda + \mathbb Z^m} = H$  and
$\overline {V + \mathbb Z^m} = \mathbb R^m.$
Then $J_1 := M(\hbox{kernel }\psi) = V \cap \mathbb Z^m$
is a rank $r$ projective subgroup of $\mathbb Z^m$ so there exists 
a rank $m-r$ subgroup $J_2$ with $J_1 \oplus J_2 = \mathbb Z^m.$ 
Therefore
\begin{equation}\label{decomp}
M\Lambda + \mathbb Z^m = \bigcup_{j \in J_2} (M\Lambda + j), \ \ \ 
V + \mathbb Z^m = \bigcup_{j \in J_2} (V + j)
\end{equation}
where the unions are disjoint. $M\Lambda \subset V$ is Delone since
$R_j(M\Lambda) = R_1(\Lambda), \, j = 1,2 .$ Define $\mu,$ $\Psi$ and $\Theta$ as in Section 2. Let $U$ be the orthognal complement of $V$ so $U \oplus V = \mathbb R^m$ and let
$P : \mathbb R^M \rightarrow U$ and $Q : \mathbb R^m \rightarrow V$ be orthogonal projections. 
\\ \\
Construct $\varphi : V + \mathbb Z^m \rightarrow M\Lambda + \mathbb Z^m$ such that 
\begin{equation}\label{varphi1}
	\varphi(v) \in M\Lambda, \  \  \ v \in V,
\end{equation}
and
\begin{equation}\label{varphi2}
	||v - \varphi(v)|| = \min \{\,  ||x - p|| \, : \, p \in M\Lambda \, \}, \ \ \ v \in V.
\end{equation}
Use the fact that $V + \mathbb Z^m$ is the disjoint union of $V + j, j \in J_2$ to extend
the domain of $\varphi$ from $V$ to $V + \mathbb Z^m$ by defining
\begin{equation}\label{varphi3}
	\varphi(v + j) = \varphi(v) + j, \  \  \ v \in V, \ j \in J_2.
\end{equation}
For $j \in J_2,$ 
$M\Lambda + j \subset V+j$ is Delone so gives a
Vorono\"{i} tesselation \cite{voronoi} of $V+j$ and $\varphi$ maps the open cell containing $p \in M\Lambda + j$ to $p.$
%
%
\begin{lem}\label{lem5}
$j \in J_2,$ $x \in V + j,$ $p \in M\Lambda + j,$ $||x-p|| < \frac{R_1(M\Lambda)}{2} \implies \varphi(x) = p.$
\end{lem}
\begin{proof}
If $q \in M\Lambda + j$ and $||x-q|| < \frac{R_1(M\Lambda)}{2}$ then $||p-q|| < R_1(M\Lambda),$ hence
$p = q$ and $p = \varphi(x).$
\end{proof} 
%
%
Choose $R < R_1(M\Lambda)/2,$ construct $g \in C_c(V)$ by
\begin{equation}\label{g}
	 g(v) := \begin{cases}
	 0, \ \ \ \ \ \ \ \ \ \ \ \ \ \ \ \ \hbox{if } ||v|| \geq R \\
	 1 - R^{-1}\, ||v||, \ \ \hbox{if } ||v|| \leq R,
	 \end{cases}
\end{equation}
and define $f := g \circ M,$ $G := (\Theta f) \circ \rho_m,$ and 
$O := \{\, x \in \mathbb R^m\, :\, G(x) > 0\, \}.$
Clearly $G : \mathbb R^m \rightarrow [0,1]$ is uniformly continuous, 
$G^{-1}(\{1\}) = H,$ and the restriction $G|_V = \widetilde {M\Lambda} g.$
%
%
\begin{lem}\label{lem6} If $x \in O \cap (V + \mathbb Z^m),$ then 
\begin{equation}\label{Gequality}
	||x - \varphi(x)|| =  R\, (1-G(x)).
\end{equation}
\end{lem}
\begin{proof}
Since $x = v +j$ for unique
$v \in V$ and $j \in J_2,$ 
$G(x) = G(v) = \sum_{p \in M\Lambda} g(v - p),$ 
so there exists a unique $p \in M\Lambda$ such that $G(v) = g(v-p).$
Since $g(v-p) > 0,$ $||v-p|| < R$ hence $p = \varphi(v)$
and $G(x) = g(v-p) = 1 - R^{-1}\,||v-p||.$ 
Therefore
$||x - \varphi(x)|| = ||v-\varphi(v)|| = ||v-p|| = R\, (1 - G(x)).$
\end{proof} 
%
%
The modulus of continuity $\omega : (0,\infty) \rightarrow [0,\infty]$ of $G,$ defined by
\begin{equation}\label{omega}
	\omega(t) := \sum \{|G(x)-G(y)|:||x-y|| < t\},
\end{equation}
satisfies $\lim_{\delta \rightarrow 0} \omega(\delta) = 0$ since $G$ is uniformly continuous.
%
%
\begin{lem}\label{lem7} $\varphi : O \cap (V + \mathbb Z^m) \rightarrow M\Lambda + \mathbb Z^m$ is uniformly continuous. 
\end{lem}
\begin{proof} 
For $i = 1, 2$ assume that $x_i = v_i + j_i \in O$ and where $v_i \in V$ and $j_i \in J_2$
and let $\Delta \in U$ satisfy $V+j_1 + \Delta = V + j_2.$ 
Since $||\Delta|| \leq ||x_2-x_1||$ and $||x_2-x_1-\Delta|| \leq ||x_2-x_1||,$  Lemma 5 implies that 
$$||x_2 - (\varphi(x_1) + \Delta)|| \leq ||x_1-\varphi(x_1)|| + ||x_2-x_1-\Delta|| \leq R\, (1-G(x_1)) + ||x_2-x_1||.$$
If $\omega(||_2-x_1||) < 1,$ then $||\varphi(x_1) + \Delta - \varphi(\varphi(x_1)+\Delta)|| \leq R\, \omega(||x_2-x_1||),$ therefore
$$||x_2 - \varphi(\varphi(x_1) + \Delta)|| \leq  R\, (1-G(x_1)) + ||x_2-x_1|| + R\, \omega(||\Delta||).$$
Lemma 4 implies that if 
$||x_2-x_1|| + \omega(||x_2-x_1||) \leq R_1(M\Lambda)/2 - R$ then
$$\varphi(\varphi(x_1)+\Delta) = \varphi(x_2),$$
hence
$$||\varphi(x_2) - \varphi(x_1)|| < R\, \omega(||x_2-x_1||) + ||x_2-x_1||,$$
which concludes the proof.
\end{proof}
%
%
%
\begin{lem}\label{le8}
$\varphi : O \rightarrow M\Lambda + \mathbb Z^m$ extends to a continuous function $\Phi : \overline O \rightarrow H.$ For every $U_v \subset U$ with $v+U_v \subset O$ there exists a continuous 
$h_v : U_v \rightarrow V$ such that
\begin{equation}\label{Uv}
\Phi(v + u) = h_v(u) + u, \ \ \ u \in U_v.
\end{equation}
\end{lem}
\begin{proof}
The first assertion follows since $\varphi$ is uniformly continuous on $O \cap (V+\mathbb Z^m)$ which is dense in $\overline O.$ Assume that $U_v \subset U$ satisfies $v+U_v \subset O$ and define $h_v(u) := \Phi(v+u)-u$ for $u \in U_v$ so Equation \ref{Uv} holds. To prove
$h_v(U_v) \subset V,$ let $u \in U_v$ and use the fact that $PJ_2$ is dense in $U$ choose a sequence $j_n \in J_2$ with $v+j_n \in O$ and
$\lim_{n \rightarrow \infty} Pj_n = u.$ Then
\begin{equation}
\Phi(v+u) = \lim_{n \rightarrow \infty} \varphi(v + Pj_n) =
\lim_{n  \rightarrow \infty} (\varphi(v-Qj_n) + Qj_n) + u.
\end{equation}
Since $\varphi(v-Qj_n) + Qj_n \in V$ and $V$ is closed, 
$h_v(u) = \lim_{n  \rightarrow \infty} (\varphi(v-Qj_n) + Qj_n) \in V.$
\end{proof} 
%
%
Let $K_c$ be a connected component of $K,$ $H_c$ be a connected component of $\rho_m^{-1}(K_c),$ and $\Lambda_c = \psi^{-1}(K_c).$
%
%
\begin{theo}\label{thm3}
There exists a continuous $h_1 : U \rightarrow V$ such that the map
$u \rightarrow h(u)+u$ is a homeomorphism of $U$ onto $H_c.$ 
\end{theo}
\begin{proof}
Clearly $H_c \subset O.$ Choose $u_1 \in U$ and $v_1 \in V$ such that
$u_1+ v_1 \in H_c.$ Choose an open bounded $U_1 \subset U$ containing $u_1$
such that $v_1 + U_1 \subset O$ and $\Phi(v_1 + U_1) \subset H_c.$ This is possible since $H_c$ is a connected component of $H.$
Then define $h_1 : \overline {U_1} \rightarrow V \cap H_c$ by 
\begin{equation}
		h_1(u) := \Phi(v_1+u) - u, \ \ \ u \in \overline {U_1}.
\end{equation}
Lemma \ref{lem7} implies that $h_1(\overline {U_1}) \subset V$ 
so the map $u \rightarrow h(u) + u$ is one-to-one and continuous  
from the compact set $\overline {U_1}$ onto its image hence is a homeomorphism. Since the point $u_1 + v_1$ was an arbitrary, $H_c$ is a manifold. The function $h_1$ can be extended to give a homeomorphism
$h_1 : U \rightarrow H_c$ to conclude the proof.
\end{proof} 
%
%
The inclusion map $i : K_c \rightarrow \mathbb T^m$ induces a map on the fundamental groups $\pi_1(i) : \pi_1(K_c) \rightarrow \pi_1(\mathbb T^m) = \mathbb Z^m.$
%
%
\begin{theo}\label{thm4}
$K_c$ is homeomorphic to $\mathbb T^{m-n}$ and $\pi_1(i)$ is injective.
\end{theo}
\begin{proof}
Define $S := \{ \, k \in \mathbb Z^m \, : \, H_c+k = H_c \, \}.$ Then $S$ is a subgroup of $\mathbb Z^m$ and it acts freely on $H_c.$ 
If $x, y \in H_c$ then $\rho_m(x) = \rho_m(y)$ iff there exists $k \in \mathbb Z^m$ such that $y = x+k.$ Since $H_c$ is connected, $k \in S.$ Therefore $\pi_1(K_c) = S,$ $\pi_1(i)$ is the inclusion map hence is injective. $K_c$ is homeomorphic to the space  of orbits of the action of $S$ on $H_c.$ Since $H_c$ is homeomorphic to $U,$ it has dimension $m-n.$ Since $K_c$ is compact and rank $S = m-n,$ $K_c$ is homeomorphic to $\mathbb T^{m-n}.$
\end{proof} 
%
The set $H_c$ is the graph of the function $h_1 : U \rightarrow V$ and for $s \in S$ the action of $s$ takes the point $h_1(u) + u \in H_c$ to 
$$(h_1(u) + Qs) + (u+Ps)$$ hence for this point to be in $H_c$ it follows that
$$h_1(u)+Qs = h_1(u+Ps).$$ Since the space of orbits in $H_c$ under the action of $S$ is homeomorphic to the compact set $K_c,$ the set $PS$ is a rank $(m-n)$ subgroup of $U.$ Let $W_S$ be the $(m-n)$--dimensional subspace of $\mathbb R^m$ spanned by elements in $S.$ 
%
%
\begin{lem}\label{lem9}
There exists a linear map $h_0 : U \rightarrow V$ such that the
map $u \rightarrow u + h_0(u)$ is a bijection of $U$ onto $W_S.$
\end{lem}
\begin{proof}
Choose a basis $s_1,...,s_{m-n} \in S$ for $W_S.$ Then $Ps_1,...,Ps_{m-n}$ is a basis for $U.$ Every $u \in U$ has a unique representation as
$c_1Ps_1+\cdots+c_{m-n}Ps_{m-n}.$ Then define 
$h_0(u) := c_1Qs_1+\cdots+c_{m-n}Qs_{m-n}.$ Then
$u+h_0(u) = c_1s_1+\cdots+c_{m-n}s_{m-n} \in W_S$
and the map $u \rightarrow u + h_0(u)$ is bijective.
\end{proof}
%
%
We observe that the smallest projective subgroup of $\mathbb Z^m$ containing $S$ is $S_1 = W_s \cap \mathbb R^m.$ Then $G := \rho_m(W_S)$ is a connected compact subgroup of $\mathbb T^m$ of dimension $m-n$ hence isomorphic to $\mathbb T^{m-n}$ and $G^\perp = S_2$ where $S_1 \oplus S_2 = \mathbb Z^m.$ The quotient group $S_1/S$ is finite and we denote its cardinality by $|S_1/S|.$ Define $\Lambda_c := \psi^{-1}(K_c)$ and the multiset $(\Lambda_0,c)$ where $\Lambda_0 := \psi^{-1}(G)$ and $c = |S_1/S|.$ Define the family of maps $F_t : U \rightarrow \mathbb R^m, \ t \in [0,1]$ by
$F_t(u) := u + (1-t)h_0(u) + th_1(u)$ and the family of subsets 
$\Lambda_t := \psi^{-1}(\rho_m(F_t(U))), \ \ t \in [0,1].$
Each $F_t$ is a continuous injection and the family of these functions is a homotopy from $W_S$ to $H_c.$ 
Therefore the family of $\rho_m \circ F_t$ gives a homotopy of the map $\rho_m : W_S \rightarrow \mathbb T^m$ to the map $\rho_m : H_c \rightarrow \mathbb T^m$ and the map $t \rightarrow \Lambda_t$ is continuous with respect to the Haussdorf metric on closed subsets of $\mathbb R^n.$ However, each fiber of the map 
$\rho_m \circ F_0 : W_S \rightarrow G$ is the orbit of the action of $S_1$ by addition, but each fiber of the map $\rho_m \circ F_1 : H_c \rightarrow K_c$ is the orbit of the action of the possibly smaller subgroup $S \subset S_1.$ Therefore the family of sets images $\rho_m \circ F_t(U)$ gives a $|S_1/S|$-to-one deformation of $K_c$ to $G.$ This argument proves that the density of $\Lambda_c$ equals the density of the multiset $(\Lambda_0,c)$ which equals $|S_1/S|$ times the density of $\Lambda_0.$ 
\begin{theo}\label{thm5}
The density of $\Lambda_c$ equals  
$|S_1/S|\, |\det E^TM|$ where $E \in \mathbb Z^{m \times n}$ satisfies $S_2 = E\, \mathbb Z^n.$
\end{theo}
\begin{proof} Itt suffices to prove that the density of $\Lambda_0$ equals
$|\det E^TM|.$ Since $G = \{z \in \mathbb T^m:\zeta_k(z) = 1 \hbox{ for all } k \in G^\perp\}$ and $G^\perp = S_2 = E\mathbb Z^n$ it follows that
$\Lambda_0 := \psi^{-1}(G) = 
\{ \lambda \in \mathbb R^n : \zeta_k \circ \rho_m (M\lambda) = 1 \hbox{ for all }
k \in E\mathbb Z^n \}$ \\ $= \{ \lambda \in \mathbb R^n : E^TM\lambda \in \mathbb Z^n \}.$ Therefore density $\Lambda_0 = |\det E^TM|.$
\end{proof}
\begin{remark}\label{rem3}
If $X$ and $Y$ are topological spaces $[X,Y]$ is the set of homotopy classes of continuous maps from $X$ to $Y.$ Let $A$ be an abelian group and $n \geq 1.$ Then $\alpha \in [X,Y]$ gives a homomorphism between cohomology groups
$H^n(\alpha) : H^n(Y;A) \rightarrow H^n(X;A).$ If $Y$ is the Eilenberg-Maclane space $K(A,n),$ then there exists special $c \in H^n(K(A,n),A)$ such that for every CW complex $X,$ the map $\alpha \rightarrow H^n(\alpha)(c) \in H^n(X;A)$ is a bijection, 
see (\cite{hatcher}, Theorem 4.57). This result implies that the map $\alpha \rightarrow \pi_1(\alpha)(\pi_1(\mathbb T^{m-n})) \subset \pi_1(\mathbb T^m)$ is a bijection between homotopy classes and subgroups of $\mathbb Z^m.$
\end{remark}
\section{Research Questions}\label{sec4}
\begin{enumerate}
\item Can the results in Section 3 be extended to locally finite multisets?
\item Example 3 includes a Besicovitch almost periodic set with a dilation and higher dimensional ones include
Penrose and Amman tilings described by Baake in \cite{baake}. A conjecture of Lagarias and Wang that we proved in
\cite{lawton} shows this is impossible if $K$ is a real analytic set. 
Do any Bohr almost periodic sets admit dilations?
\item Does this conjecture hold: $\Lambda \subset \mathbb R^n$ is a FQ 
iff there exists a toral compactification $(\mathbb T^m,\psi),$ $p_1,...,p_n \in T(\mathbb T^m),$ and
$\Lambda$ is the set of common zeros in $\mathbb C^n$ of the entire extensions $P_j$ of $p_j \circ \psi$ for $j = 1,...,n?$ Theorem \ref{thm1} validates this conjecture for $n = 1.$ We suggest deriving a Poisson summation formula, as
Lev and Olevskii did in \cite{levolevskii1} for $n = 1,$ using multidimension residue methods described by Tsikh \cite{tsikh} to compute an integral 
$\int_{\mathbb R^n} \frac{F}{P_1\cdots P_n}$ 
for a class of entire functions $F.$  We record that the hypothesis of the conjecture implies the map $ : \mathbb C^n \rightarrow \mathbb C^n$ defined by $(z_1,...,z_n) \rightarrow (P_1(z_1,...,z_n),...,P_n(z_1,...,z_n))$ satisfies conditions ensuring the existence of local Grothendieck residues (\cite{tsikh}, p. 14).
\end{enumerate}
{\bf Acknowledgments} The author thanks Alexander Olevskii, Yves Meyer and Peter Sarnak for sharing their knowledge of crystaline measures and Fourier quasicrystals and August Tsikh for suggesting multidimensional residues.

\Finishall
\end{document}